\documentclass[reqno]{amsart}
\usepackage{amsmath, amssymb, amsthm, epsfig}
\usepackage{hyperref, latexsym}
\usepackage{url}
\usepackage[mathscr]{euscript}
\usepackage{enumerate}
\usepackage{color}
\usepackage{fullpage} 
\usepackage{setspace}

\def\today{\ifcase\month\or
  January\or February\or March\or April\or May\or June\or
  July\or August\or September\or October\or November\or December\fi
  \space\number\day, \number\year}

 \newtheorem{theorem}{Theorem}
  
 \newtheorem{lemma}[theorem]{Lemma}
 
 \newtheorem{corollary}[theorem]{Corollary}
 \theoremstyle{definition}

 \theoremstyle{remark}

 \newcommand{\C}{\mathbb{C}}
 \newcommand{\R}{\mathbb{R}}

 \newcommand{\hn}{\tfrac32}
 
  \newcommand{\hh}{\tfrac12}

  \renewcommand{\d}{\text{\rm d}}
 \newcommand{\du}{\text{\rm d}u}

 \newcommand{\dx}{\text{\rm d}x}

\newcommand{\im}{{\rm Im}\,}
\newcommand{\re}{{\rm Re}\,}

\newcommand{\dd}{\,{\rm d}}

\begin{document}
\title[A note on entire $L$-functions]{A note on entire $L$-functions}
\author[Chirre]{Andr\'{e}s Chirre}
\address{IMPA - Instituto Nacional de Matem\'{a}tica Pura e Aplicada - Estrada Dona Castorina, 110, Rio de Janeiro, RJ, Brazil 22460-320}
\email{achirre@impa.br}

\allowdisplaybreaks
\numberwithin{equation}{section}

\maketitle

\begin{abstract}
	In this paper, we exhibit upper and lower bounds with explicit constants for some objects related to entire $L$-functions in the critical strip, under the generalized Riemann hypothesis. The examples include the entire Dirichlet $L$-functions $L(s,\chi)$ for primitive characters $\chi$.
\end{abstract}

\section{Introduction}

Recently, new estimates for some objects related to $L$-functions have been given. In particular, we have estimates for the Riemann zeta-function  \cite{CC, CCM, CCM2, CChi, CChiM, CS} under the Riemann hypothesis. For a general family of $L$-functions in the framework of \cite[Chapter 5]{IK}, we have similar estimates in the critical line \cite{CCM2, CChi, CF, CS} under the generalized Riemann hypothesis. The purpose of this paper is to exhibit explicit bounds for a family of entire $L$-functions in the critical strip. The proof of these estimates is motivated by the ideas of Carneiro and Chandee \cite{CC}, and Carneiro, Chirre and Milinovich \cite{CChiM} on the use of the Guinand-Weil explicit formula applied to special functions with compactly supported Fourier transforms.

\subsection{Entire $L$-functions} In this paper we study a family of entire $L$-functions that includes the Dirichlet series $L(s,\chi)$ for non-principal primitive characters $\chi$. Similar families of $L$-functions are studied in \cite{MG, NP}. We adopt the notation
$$\Gamma_\mathbb R(s):=\pi^{-s/2}\,\Gamma\left(\frac{s}{2}\right),$$ 
where $\Gamma$ is the usual Gamma function. Throughout this paper we consider that an entire function $L(s,\pi)$ meets the following requirements (for some positive integer $d$):

\medskip

\noindent(i) There exists a sequence $\{\lambda_\pi(n)\}_{n\ge1}$ of complex numbers ($\lambda_\pi(1) =1$) such that the series $$\sum_{n=1}^\infty\frac{\lambda_\pi(n)}{n^s}$$ converges absolutely to $L(s,\pi)$ on $\{s\in\mathbb C \,;\,\text{Re}\,s>1\}$.

\medskip

\noindent(ii) For each prime number $p$, there exist $\alpha_{1,\pi}(p),\alpha_{2,\pi}(p),\ldots,\alpha_{d,\pi}(p)$ in $\mathbb C$ such that $|\alpha_{j,\pi}(p)|\leq1$, and 
$$L(s,\pi)=\prod_p\prod_{j=1}^d\left(1-\frac{\alpha_{j,\pi}(p)}{p^s}\right)^{-1},$$
where the infinite product converges absolutely on $\{s \in \C; \text{Re}\,s>1\}$.

\medskip

\noindent(iii) For some positive integer $N$ and some complex numbers $\mu_1,\mu_2,\ldots,\mu_d$ with nonnegative real parts and such that $\{\mu_1,\mu_2,\ldots,\mu_d\}=\{\overline{\mu_1},\overline{\mu_2},\ldots,\overline{\mu_d}\}$, the complete $L$-function 
$$\Lambda(s,\pi):=N^{s/2}\prod_{j=1}^d \Gamma_\mathbb R(s+\mu_j)L(s,\pi)$$
is an entire function of order 1 having no zeros in $0$ and $1$. Furthermore, the function $\Lambda(s, \tilde\pi):=\overline{\Lambda(\overline s, \pi)}$ satisfies the functional equation
\begin{equation*}
\Lambda(s,\pi)=\kappa\,\Lambda(1-s,\tilde\pi),
\end{equation*}
for some complex number $\kappa$ of norm $1$. 

\smallskip

 Now, by using the product expansion of $L(s,\pi)$ and the inequality $|\alpha_{j,\pi}(p)|\leq 1$ we obtain that \footnote{Throughout the paper we use the notation $f \ll g$ to mean that for a certain constant $C>0$ we have $f(t) \leq Cg(t)$ for $t\in Dom(f) \cap Dom(g)$. In the subscript we indicate the parameters in which such constant $C$ may depend on.}
\begin{equation}\label{L_fun_eq1}
|\log|L(s,\pi)||\leq d\log\zeta({\rm Re}\,s)\ll\dfrac{d}{2^{\re{s}}}
\end{equation} 
for any $s$ with ${\rm Re}\,s\geq\frac{3}{2}$. Besides, we have that \begin{equation}\label{L_fun_eq2}
	\frac{L'}{L}(s,\pi)=-\sum_{n=2}^\infty\frac{\Lambda_\pi(n)}{n^s}
\end{equation}
converges absolutely if ${\rm Re}\,s>1$, and $\Lambda_\pi(n)=0$ if $n$ is not a power of prime and $\Lambda_\pi(p^k)=\sum_{j=1}^d\alpha_{j,\pi}(p)^k\log p$ if $p$ is prime and $k$ is a positive integer. Thus
\begin{equation}\label{1L_fun_eq3}
	\big|\Lambda_\pi(n)\big| \leq d \,\Lambda(n).
\end{equation}

\bigskip

\section{Main results} Let $n\geq 0$ be an integer, $\hh\leq\sigma\leq1$ be a real parameter, and $L(s,\pi)$ be an entire $L$-function in the above setting. For $t\in\R$ (and $t$ not coinciding with the ordinate of a zero of $L(s,\pi)$ when $n=0$) we define the {\it iterates of the argument function} as
\[
S_{n,\sigma}(t,\pi):=-\dfrac{1}{\pi}\im\bigg\{\dfrac{i^n}{n!}\int_{\sigma}^{\infty}(u-\sigma)^n\, \dfrac{L'}{L}(u+it,\pi)\,\du\bigg\}.
\]
If $t$ is the ordinate of a zero of $L(s,\pi)$ when $n=0$ we define
\[
S_{0,\sigma}(t,\pi):=\displaystyle\lim_{\varepsilon\to 0}\dfrac{S_{0,\sigma}(t+\varepsilon,\pi)+S_{0,\sigma}(t-\varepsilon,\pi)}{2}.
\]
Differentiating under the integral sign and using integration by parts, one can see that $S'_{n,\sigma}(t,\pi)= S_{n-1,\sigma}(t,\pi)$ for $t\in\R$ (in the case $n=1$ we may restrict ourselves to the case when t is not the ordinate of a zero of $L(s,\pi)$). We finally define  
\begin{equation*}
S_{-1,\sigma}(t,\pi): = \dfrac{1}{\pi}\re{\dfrac{L'}{L}(\sigma+it,\pi)},
\end{equation*}
when $t$ is not the ordinate of a zero of $L(s,\pi)$. We can see that $S'_{0,\sigma}(t,\pi)=S_{-1,\sigma}(t,\pi)$, when $t$ is not the ordinate of a zero of $L(s,\pi)$.

\medskip

Theorem \ref{teo_prin} below provides estimates for the above mentioned objects and for the logarithm of the modulus of $L(s,\pi)$ in the critical strip. These results are based on the generalized Riemann hypothesis, which states that $\Lambda(s,\pi)\neq 0$ if $\re{s}\neq \hh$. As in \cite{CCM2, CChi, CF, CS}, the {\it analytic conductor} of $L(s,\pi)$, which is defined by $$C(t,\pi)=N\prod_{j=1}^d(|it+\mu_j|+3),$$
will appear in our results. For an integer $n\geq0$ we introduce the function
\begin{equation*}
H_n(x):=\displaystyle\sum_{k=0}^{\infty}\dfrac{x^k}{(k+1)^n}.
\end{equation*}
In particular, when $0<|x|<1$ we have that 
\begin{align} \label{25_10_1:52am} 
\dfrac{\log(1\pm x)}{x} = \pm H_1(\mp x).
\end{align}

\medskip

\begin{theorem}  \label{teo_prin} Let $L(s,\pi)$ be an entire $L$-function satisfying the generalized Riemann hypothesis. Let $c>0$ be a given real number. Then, for $\hh<\sigma<1$ and $t\in\R$ in the range  $$(1-\sigma)^2\log\log C(t,\pi)\geq c,$$ we have the following uniform bounds:
\smallskip

\noindent{\rm (i)} For the logarithm, \footnote{Throughout the paper we use the notation $f = O(g)$ to mean that for a certain constant $C>0$ we have $|f(t)|\leq Cg(t)$ for $t$ sufficiently large. In the subscript we indicate the parameters in which such constant $C$ may depend on.}
	\begin{align*}
	-M^{-}_{\sigma}(t)\dfrac{(\log C(t,\pi))^{2-2\sigma}}{\log\log C(t,\pi)} +  & O_c\bigg(\dfrac{d\,\mu(\sigma)\,(\log C(t,\pi))^{2-2\sigma}}{(1-\sigma)^{2}(\log\log C(t,\pi))^2}\bigg) \leq \log|L(\sigma+it,\pi)| \\
	& \leq  
	M^{+}_{\sigma}(t)\dfrac{(\log C(t,\pi))^{2-2\sigma}}{\log\log C(t,\pi)} +  O_c\bigg(\dfrac{d\,(\log C(t,\pi))^{2-2\sigma}}{(1-\sigma)^{2}(\log\log C(t,\pi))^2}\bigg).
	\end{align*}
\noindent{\rm (ii)}	For $n\geq -1$ an integer,
	\begin{align*}
	-M^{-}_{n,\sigma}(t)\dfrac{(\log C(t,\pi))^{2-2\sigma}}{(\log\log C(t,\pi))^{n+1}} +  & O_c\bigg(\dfrac{d\,\mu^{-}_{n,d}(\sigma)\,(\log C(t,\pi))^{2-2\sigma}}{(1-\sigma)^{2}(\log\log C(t,\pi))^{n+2}}\bigg)\leq S_{n,\sigma}(t,\pi) \\
	& \leq  
	M^{+}_{n,\sigma}(t)\dfrac{(\log C(t,\pi))^{2-2\sigma}}{(\log\log C(t,\pi))^{n+1}} +  O_c\bigg(\dfrac{d\,\mu^{+}_{n,d}(\sigma)\,(\log C(t,\pi))^{2-2\sigma}}{(1-\sigma)^{2}\,(\log\log C(t,\pi))^{n+2}}\bigg).
	\end{align*}
The functions appearing above are given by
\begin{itemize}
	\item For the logarithm,
	\begin{equation*}
	M_{\sigma}^{\pm}(t) = \frac{1}{2}\left( H_{1}\Big(\mp(\log C(t,\pi))^{1 - 2\sigma}\Big) + \frac{d\,(2\sigma -1)}{\sigma (1- \sigma)}\right) \hspace{0.2cm} \mbox{and}\hspace{0.2cm} \mu(\sigma)=\dfrac{|\log(\sigma-\hh)|}{\sigma-\hh}.
	\end{equation*}
	\item For $n\geq1$ odd,
	\begin{equation*}
	M_{n,\sigma}^{\pm}(t) = \frac{1}{2^{n+1}\pi}\left( H_{n+1}\Big(\pm (-1)^{\frac{(n+1)}{2}}(\log C(t,\pi))^{1 - 2\sigma}\Big) + \frac{d\,(2\sigma -1)}{\sigma (1- \sigma)}\right)\hspace{0.2cm} \mbox{and}\hspace{0.2cm} \mu^{\pm}_{n,d}(\sigma)=1.
	\end{equation*}
	\item For  $n=-1$,
	\begin{equation*}
	M_{-1,\sigma}^{\pm}(t) = \frac{1}{\pi}\left( H_{0}\Big(\pm (\log C(t,\pi))^{1 - 2\sigma}\Big) + \frac{d\,(2\sigma -1)}{\sigma (1- \sigma)}\right) \hspace{0.2cm} \mbox{and}\hspace{0.2cm} \mu^{\pm}_{-1,d}(\sigma)=(\sigma-\hh)^{\mp 1}.
	\end{equation*}
	\item For $n =0$, 
	\begin{equation*}
	M_{0,\sigma}^{\pm}(t) =  \Big(2 \big(M_{1,\sigma}^{+}(t) + M_{1,\sigma}^{-}(t)\big) \, M_{-1,\sigma}^{-}(t)\Big)^{\hh}\hspace{0.2cm} \mbox{and}\hspace{0.2cm} \mu^{\pm}_{n,d}(\sigma)=(2\sigma-1)d + 1.
	\end{equation*}
	\item For $n\geq2$ even,
\begin{equation*}
M_{n,\sigma}^{\pm}(t) = \left(\frac{2 \big(M_{n+1,\sigma}^{+}(t) + M_{n+1,\sigma}^{-}(t)\big) \, M_{n-1,\sigma}^{+}(t)\, M_{n-1,\sigma}^{-}(t)}{M_{n-1,\sigma}^{+}(t) + M_{n-1,\sigma}^{-}(t)}\right)^{\hh}\hspace{0.2cm} \mbox{and}\hspace{0.2cm} \mu^{\pm}_{n,d}(\sigma)=(2\sigma-1)d + 1.
\end{equation*}
\end{itemize}
\end{theorem}

\smallskip

When $\sigma \to \hh$ in the above theorem we obtain a sharpened version of the results in \cite{CCM2, CChi, CF, CS} for the case of entire $L$-functions with improved error terms (a factor $\log\log\log C(t,\pi)^{\frac{3}{d}}$ has been removed). Furthermore, for a fixed $\hh<\sigma<1$ we obtain bounds as $C(t,\pi)\to \infty$.

\begin{corollary}
Let $L(s,\pi)$ be an entire $L$-function satisfying the generalized Riemann hypothesis and let $n \geq -1$. Let $\hh < \sigma < 1$ be a fixed number. Then
\begin{equation*}
\log|L(\sigma+it,\pi)| \leq \frac{1}{2}\left(1 + o(1) + d\,\bigg(\frac{2\sigma -1}{\sigma(1-\sigma)} + o(1)\bigg) \right)  \frac{(\log C(t,\pi))^{2-2\sigma}}{\log \log C(t,\pi)},
\end{equation*}
and
\begin{equation*}
|S_{n,\sigma}(t)| \leq \frac{\omega_n}{2^{n+1}\pi}\left(1 + o(1) + d\,\bigg(\frac{2\sigma -1}{\sigma(1-\sigma)} + \mu_{d,\sigma}\,o(1)\bigg) \right)  \frac{(\log C(t,\pi))^{2-2\sigma}}{(\log \log C(t,\pi))^{n+1}}
\end{equation*}
as $C(t,\pi) \to \infty$, where $\omega_n = 1$ and $\mu_{d,\sigma}=1$ if $n$ is odd, and $\omega_n = \sqrt{2}$ and $\mu_{d,\sigma}=(2\sigma-1)d + 1$ if $n$ is even. 
\end{corollary}

\bigskip

\section{Preliminaries}

The proof of Theorem \ref{teo_prin} follows the same circle of ideas used to prove estimates for the Riemann zeta-function in \cite{CChiM, CS}. First, we show the results for $\log|L(s,\pi)|$ and $S_{n}(t,\pi)$, when $n\geq -1$ odd. In these cases, we need three ingredients: a suitable representation lemma for our objects, the Guinand-Weil explicit formula connecting primes and zeros and some extremal bandlimited approximations. 

\subsection{Representation Lemma} The idea of the representation lemma is to have formulas of the objects to be bounded, assuming the generalized Riemann hypothesis. Let $m\geq 0$ be an integer and $\hh<\sigma\leq1$ be a real number. Consider the functions $f_{\sigma},f_{2m+1,\sigma},f_{1,\sigma}:\R \to \R$ defined by

\begin{equation*}\label{Def_f}
f_{\sigma}(x)=\log\Bigg(\dfrac{1+x^2}{\big(\sigma-\hh\big)^2+x^2}\Bigg),
\end{equation*} 
\begin{equation*}\label{Def_f_2m+1}
f_{2m+1,\sigma}(x)=\frac{1}{2}\int_{\sigma}^{\tfrac{3}{2}}{(u-\sigma)^{2m}\,\log\left(\dfrac{1+x^2}{(u-\hh)^2+x^2}\right)} \,\du,
\end{equation*} 
and  
\begin{equation*}\label{Def_f_-1}
f_{-1,\sigma}(x)= \frac{(\sigma - \hh)}{(\sigma - \hh)^2 + x^2}.
\end{equation*}

\medskip

Lemma \ref{lem_rep} has appeared in \cite[Lemma 7]{CChiM} in the case $\sigma=\hh$ for the Riemann zeta-function. The proof for entire $L$-functions follows the same outline (see \cite[Lemma 4]{CF}).

\begin{lemma} \label{lem_rep} Let $L(s,\pi)$ be an entire $L$-function satisfying the generalized Riemann hypothesis and $m\geq0$ be an integer.  Then, for $\hh<\sigma\leq 1$ and $t\in\mathbb{R}$ we have
	\begin{equation} \label{Rep_lem_f}
	\log|L(\sigma+it,\pi)|= \bigg(\dfrac{3}{4}-\dfrac{\sigma}{2}\bigg)\log C(t,\pi)-\dfrac{1}{2}\displaystyle\sum_{\gamma}f_{\sigma}(t-\gamma)  %-\displaystyle\sum_{-1<\re\mu_j\leq 0}\log\bigg|\dfrac{\hn+it+\mu}{s+\mu}\bigg| 
	+ O(d), 
	\end{equation}
	\begin{equation} \label{Rep_lem_f_2}
	S_{2m+1,\sigma}(t,\pi)=\dfrac{(-1)^m}{2\pi(2m+2)!}\,\bigg(\dfrac{3}{2}-\sigma\bigg)^{2m+2}\log C(t,\pi) -\dfrac{(-1)^m}{\pi(2m)!}\,\displaystyle\sum_{\gamma}f_{2m+1,\sigma}(t-\gamma)+ O_m(d),
	\end{equation}
	\begin{equation} \label{Rep_lem_f_{-1}}
	S_{-1,\sigma}(t,\pi)=-\dfrac{\log C(t,\pi)}{2\pi}+\dfrac{1}{\pi}\displaystyle\sum_{\gamma}f_{-1,\sigma}(t-\gamma)+ O(d),
	\end{equation}
where the sums run over all values of $\gamma$ such that $L\big(\hh+i\gamma,\pi\big)=0$, counted with multiplicity.
\end{lemma}
\begin{proof} First, we prove \eqref{Rep_lem_f}. For $\hh\leq\sigma\leq\tfrac{3}{2}$ we have that
	\begin{align} \label{lem_G_1}
	\log\bigg|\dfrac{L(\sigma+it,\pi)}{L(\hn+it,\pi)}\bigg| = \log\bigg|\dfrac{\Lambda(\sigma+it,\pi)}{\Lambda(\hn+it,\pi)}\bigg| + \log\bigg|\dfrac{N^{(3/2+it)/2}}{N^{(\sigma+it)/2}}\bigg|+ \displaystyle\sum_{j=1}^{d}\log\bigg|\dfrac{\Gamma_{\mathbb{R}}\big(\hn+it+\mu_j\big)}{\Gamma_{\mathbb{R}}(\sigma+it+\mu_j)}\bigg|.
	\end{align}
	%usa expansion del producto, luego que |ln(1-z)|\leqC|z|..
We treat each term on the right-hand side of \eqref{lem_G_1}. From Hadamard's factorization formula \cite[Theorem 5.6 and Eq. (5.29)]{IK}, the analyticity of $L(s,\pi)$ and the generalized Riemann hypothesis, it follows that 
	\begin{align} \label{lem_G_3}
	\log\bigg|\dfrac{\Lambda(\sigma+it,\pi)}{\Lambda \big(\hn+it,\pi\big)}\bigg|&=-\dfrac{1}{2}\displaystyle\sum_{\gamma}\log\Bigg(\dfrac{1+(t-\gamma)^2}{\big(\sigma-\hh\big)^2+(t-\gamma)^2}\Bigg),
	\end{align} 
	where the sums run over all values of $\gamma$ such that $\Lambda\big(\hh+i\gamma,\pi\big)=0$, counted with multiplicity. A simple computation of the second term show that
	\begin{align} \label{lem_G_4}
	\log\bigg|\dfrac{N^{(3/2+it)/2}}{N^{(\sigma+it)/2}}\bigg|=\bigg(\dfrac{3}{4}-\dfrac{\sigma}{2}\bigg)\log N. 
	\end{align}
	To analyze the third term, we shall use the Stirling's formula in the form
    \begin{align} \label{New_Stirling}
	\dfrac{\Gamma_{\R}'}{\Gamma_{\R}}(s)=\dfrac{1}{2}\log s+O(1),
	\end{align}
	which is valid for $\re{s}\geq\hh$. Since $\re{\mu_j}\geq 0$, we have 
	\begin{equation} \label{15_6_11:40pm}
	\re\dfrac{\Gamma'_{\R}}{\Gamma_{\R}}(u+\mu_j+it) = \dfrac{1}{2}\log(|\mu_j+it|+3) + O(1)
	\end{equation}
	uniformly in $\hh\leq u \leq \frac{3}{2}$, so that	
	\begin{align} 
	\begin{split}  \label{lem_G_5}
	\log\bigg|\dfrac{\Gamma_{\mathbb{R}}\big(\hn+it+\mu_j\big)}{\Gamma_{\mathbb{R}}(\sigma+it+\mu_j)}\bigg|& = \re\int_{\sigma}^{\hn}(\log\Gamma_{\R}(u+\mu_j+it))'\,\du=\int_{\sigma}^{\hn}\re\dfrac{\Gamma'_{\R}}{\Gamma_{\R}}(u+\mu_j+it)\,\du  \\
	&=\bigg(\dfrac{3}{4}-\dfrac{\sigma}{2}\bigg)\log(|\mu_j+it|+3)+O(1).
	\end{split}
	\end{align}
For the left-hand side of \eqref{lem_G_1}, using \eqref{L_fun_eq1} we get
\begin{align} \label{3_10_7:06pm}
\log|L(\hn+it,\pi)|=O(d).
\end{align}
Finally, using \eqref{lem_G_3}, \eqref{lem_G_4}, \eqref{lem_G_5} and \eqref{3_10_7:06pm} in \eqref{lem_G_1} we obtain for $\hh\leq\sigma\leq\tfrac{3}{2}$ and $t\in\R$ that
 	\begin{align} \label{28_9_2:7am}
	\log|L(\sigma+it,\pi)|= \bigg(\dfrac{3}{4}-\dfrac{\sigma}{2}\bigg)\log C(t,\pi)-\dfrac{1}{2}\displaystyle\sum_{\gamma}\log\Bigg(\dfrac{1+(t-\gamma)^2}{\big(\sigma-\hh\big)^2+(t-\gamma)^2}\Bigg) + O(d).
	\end{align}
	 This yields the desired result. In order to prove \eqref{Rep_lem_f_2}, we use integration by parts and \eqref{L_fun_eq1} to get
	\begin{align} \label{29_10_4:24am}
	S_{2m+1,\sigma}(t,\pi) & = \dfrac{(-1)^{m}}{\pi(2m)!}\,\Bigg\{\int_{\sigma}^{\hn}(u-\sigma)^{2m}\,\log |L(u+it,\pi)|\du\Bigg\} + O_m(d).
    \end{align}
	Then, inserting \eqref{28_9_2:7am} in \eqref{29_10_4:24am} and straightforward computations will imply \eqref{Rep_lem_f_2}. %that 
	%\begin{align*}
	%S_{1,\sigma}(t,\pi) & = \frac{1}{\pi} \int_{\sigma}^{\hn}\bigg(\dfrac{3}{4}-\dfrac{u}{2}\bigg)\,\du\,\log C(t,\pi) -\frac{1}{2\pi }\,\int_{\sigma}^{\hn}\displaystyle\sum_{\gamma}\log\Bigg(\dfrac{1+(t-\gamma)^2}{\big(u-\hh\big)^2+(t-\gamma)^2}\Bigg)\,\du + O(d) \\
	%& = \dfrac{1}{4\pi}\bigg(\dfrac{3}{2}-\sigma\bigg)^{2}\log C(t,\pi)-\frac{1}{\pi}\,\displaystyle\sum_{\gamma}\dfrac{1}{2}\int_{\sigma}^{\hn}\log\Bigg(\dfrac{1+(t-\gamma)^2}{\big(u-\hh\big)^2+(t-\gamma)^2}\Bigg)\,\du + O(d), 
	%\end{align*}
	%where the interchange between summation and integration can be justified by the monotone convergence theorem. 
	Finally, we prove \eqref{Rep_lem_f_{-1}}. By the partial fraction descomposition of the logarithmic derivative of $L(s,\pi)$ in \cite[Theorem 5.6]{IK}, we have
	\begin{equation*}
	\dfrac{L'}{L}(\sigma+it,\pi)=\displaystyle\sum_{\rho}\bigg(\dfrac{1}{\sigma+it-\rho}+\dfrac{1}{\rho}\bigg)+B-\dfrac{\log N}{2}-\displaystyle\sum_{j=1}^d\dfrac{\Gamma'_\R}{\Gamma_\R}(\sigma+it+\mu_j),
	\end{equation*}
	where $\re{B}=-\re{\sum_{\rho}\rho^{-1}}$. Then, taking the real part of this equation, considering that $\rho=\frac{1}{2}+i\gamma$ and using \eqref{15_6_11:40pm} we obtain \eqref{Rep_lem_f_{-1}} as required. 	
\end{proof}

\subsection{Guinand-Weil explicit formula}
Note that in the above representations each object is written as a sum of a translate of some function of a real variable over the non-trivial zeros of $L(s,\pi)$ plus some known terms and a small error. 
A useful tool one can use to evaluate sums over the non-trivial zeros of $L(s,\pi)$ is the Guinand-Weil explicit formula. In our setting of entire $L$-functions we shall use the following version (the proof of the general version can be found in \cite[Lemma 5]{CF}).
\begin{lemma} \label{Exp_for_L} Let $L(s,\pi)$ be an entire $L$-function. Let $h(s)$ be analytic in the strip $|\im{s}|<\tfrac12+\varepsilon$ for some $\varepsilon>0$, and assume that $|h(s)|\ll(1+|s|)^{-(1+\delta)}$ for some $\delta>0$ when $|\re{s}|\to\infty$. Then 
	\begin{align*}
	\begin{split}
	\sum_{\rho} h\left(\dfrac{\rho-\hh}{i}\right)&= \frac{\log N}{2\pi}\widehat{h}(0) + \frac{1}{\pi}\sum_{j=1}^d\int_{-\infty}^\infty h(u)\,{\rm Re}\,\frac{\Gamma_\mathbb R'}{\Gamma_\mathbb R}\left(\hh+\mu_j+iu\right)\d u\\
	& \ \ \ -\frac{1}{2\pi}\sum_{n=2}^\infty\frac{1}{\sqrt{n}}\left\{\Lambda_{\pi}(n)\, \widehat h\left(\frac{\log n}{2\pi}\right)+\overline{\Lambda_{\pi}(n)}\, \widehat h\left(\frac{-\log n}{2\pi}\right)\right\},
	\end{split}
	\end{align*}
	where the sum runs over all zeros $\rho$ of $\Lambda(s, \pi)$ and the coefficients $\Lambda_\pi(n)$ are defined by \eqref{L_fun_eq2}.
\end{lemma}

\subsection{Extremal functions}

Observe that the functions $f_{\sigma}, f_{2m+1,\sigma}$ and $f_{-1,\sigma}$ do not verify the required smoothness properties to apply the Guinand-Weil formula. Then, we replace each of these functions by appropriate extremal majorants and minorants of exponential type (thus with a compactly supported Fourier transform by the Paley-Wiener theorem), that minimize the $L^{1}(\R)$-distance. These extremal functions may be found by means of the Gaussian subordination framework of Carneiro, Littmann and Vaaler \cite{CLV}. The following lemma shows some properties of the extremal functions for $f_{\sigma}$. The proof of this result follows from \cite[Lemma 3.2]{C} (see also \cite[Lemma 5-8]{CC}).

\begin{lemma}\label{Ext_func_f} Let $\hh<\sigma< 1$ and $\Delta\geq 0.02$ be real numbers and let $\Omega(\sigma)=|\log(\sigma-\hh)|$. Then there is a pair of real entire functions $g^{\pm}_{\sigma,\Delta}:\mathbb{C}\to\mathbb{C}$ satisfying the following properties:
	\begin{itemize}
		\item[(i)] For $x\in\mathbb{R}$ we have
		\begin{align}\label{Ext_func_f_1}
		-\dfrac{1}{1+x^2} \ll g^{-}_{\sigma,\Delta}(x) \leq f_{\sigma}(x) \leq  g^{+}_{\sigma,\Delta}(x) \ll \dfrac{\Omega(\sigma)}{(\sigma-\hh)^2+x^2}. 
		\end{align}
		Moreover, for any complex number $z=x+iy$ we have
		\begin{align}\label{Ext_func_f_21}
		\big|g^{-}_{\sigma,\Delta}(z)\big|\ll \dfrac{\Delta^{2}e^{2\pi\Delta|y|}}{(1+\Delta|z|)},  
		\end{align}
		and
		\begin{align}\label{Ext_func_f_22}
		\big|g^{+}_{\sigma,\Delta}(z)\big|\ll \dfrac{\Omega(\sigma)\Delta^{2}e^{2\pi\Delta|y|}}{(1+\Delta|z|)}.  
		\end{align}
		
		\smallskip
		
		\item[(ii)] The Fourier transforms of $m^{\pm}_{\Delta}$ are even continuous functions supported on the interval $[-\Delta,\Delta]$. For $0<\xi<\Delta$ these are given by
		\begin{align} \label{27_9_1:55am}
		\begin{split}  \widehat{g}^{\pm}_{\sigma,\Delta}(\xi) = \sum_{k=-\infty}^{\infty}(\pm 1)^k \frac{(k+1)}{|\xi+k\Delta|}\Big(e^{-2\pi |\xi+k\Delta|(\sigma-\frac12)}- e^{-2\pi|\xi+k\Delta|}\Big).
		\end{split}
		\end{align}
		\smallskip
		\item [(iii)] At $\xi=0$ we have
		\begin{align} \label{Ext_func_G_41}
		\widehat{g}^{\pm}_{\sigma,\Delta}(0)=2\pi\bigg(\dfrac{3}{2}-\sigma\bigg)-\dfrac{2}{\Delta}\log\bigg(\dfrac{1\mp e^{-(2\sigma-1)\pi\Delta}}{1\mp e^{-2\pi\Delta}}\bigg).
		\end{align}
	\end{itemize}
\end{lemma}

\bigskip

Analogously, the  next lemma gives some properties of the extremal functions for $f_{2m+1,\sigma}$. The proof of this result follows from \cite[Lemma 10]{CChiM}.

\begin{lemma}\label{Rep_lem_f_{-1}_2} Let $m\geq 0$ be an integer and let $\hh<\sigma< 1$ and $\Delta\geq 0.02$ be real numbers. Then there is a pair of real entire functions $g_{2m+1,\sigma,\Delta}^{\pm}:\mathbb{C}\to\mathbb{C}$ satisfying the following properties:
	\begin{itemize}
		\item[(i)] For $x\in\R$ we have
		\begin{align}\label{Ext_func_f_1_2_2}
		-\dfrac{1}{1+x^2}\ll_m g_{2m+1,\sigma,\Delta}^{-}(x) \leq f_{2m+1,\sigma}(x) \leq g_{2m+1,\sigma,\Delta}^{+}(x) \ll_m \dfrac{1}{1+x^2}.
		\end{align}
		Moreover, for any complex number $z=x+iy$ we have
		\begin{align}\label{bound_g_2m+1_complex}
		\big|g_{2m+1,\sigma,\Delta}^{\pm}(z)\big|\ll_m\dfrac{\Delta^{2}e^{2\pi\Delta|y|}}{(1+\Delta|z|)}.
		\end{align}
		
		\smallskip
		
		\item[(ii)] The Fourier transforms of $g^{\pm}_{2m+1,\sigma,\Delta}$ are even continuous functions supported on the interval $[-\Delta,\Delta]$. For $0<\xi<\Delta$ these are given by
		\begin{align}\label{27_9_1:55am_2_2_2}
		\begin{split}
				&\widehat{g}^{\pm}_{2m+1,\sigma,\Delta}(\xi)  =\\
				&   \ \ \ \dfrac{1}{2}\sum_{k=-\infty}^{\infty}(\pm 1)^k \left[\frac{k+1}{|\xi+k\Delta|}\Bigg(\frac{(2m)!\,e^{-2\pi |\xi+k\Delta|(\sigma-\frac12)}}{(2\pi|\xi+k\Delta|)^{2m+1}}- \!\!\sum_{j=0}^{2m+1}\frac{\gamma_{j}\,e^{-2\pi|\xi+k\Delta|}}{(2\pi |\xi+k\Delta|)^{j}}\left(\dfrac{3}{2}-\sigma\right)^{2m+1-j}\Bigg)\right]\,,
			\end{split}
		\end{align}
		\smallskip
		where $\gamma_{j}=\frac{(2m)!}{(2m+1-j)!}$, for $0\leq j \leq 2m+1$.

		%\begin{align}\label{FT_maj_general_case}
		%\begin{split}
		%&\widehat{g}^{\pm}_{1,\sigma,\Delta}(\xi)  = \dfrac{1}{2}\sum_{k=-\infty}^{\infty}(\pm 1)^k \left[\frac{k+1}{|\xi+k\Delta|}\Bigg(\frac{e^{-2\pi |\xi+k\Delta|(\sigma-\frac12)}}{2\pi|\xi+k\Delta|}-\left(\dfrac{3}{2}-\sigma\right)e^{-2\pi|\xi+k\Delta|}-\frac{\,e^{-2\pi|\xi+k\Delta|}}{(2\pi |\xi+k\Delta|)}\Bigg)\right].
		%\end{split}
		%\end{align}
		
		\smallskip
		
		\item[(iii)] At $\xi=0$ we have
		\begin{align}\label{27_9_1:55am_2_2}
		\widehat{g}^{\pm}_{2m+1,\sigma,\Delta}(0)=\dfrac{\pi}{(2m+1)(2m+2)}\left(\dfrac{3}{2}-\sigma\right)^{2m+2}-\dfrac{1}{\Delta}\int_{\sigma}^{\hn}(u-\sigma)^{2m}\,\log\left(\dfrac{1\mp e^{-2\pi(u-\frac12)\Delta}}{1\mp e^{-2\pi\Delta}}\right)\du.
	\end{align} 
	\end{itemize}
\end{lemma}

\bigskip

Finally, the following lemma shows some properties of the extremal functions for $f_{-1,\sigma}$. The proof of this result follows from \cite[Lemma 9]{CChiM}. To simplify the notation we let $\beta=\sigma-\hh$.

\begin{lemma}\label{Ext_func_f_{-1}} For $0<\beta<\hh$, we define the function
	\begin{align*}
	h_{\beta}(x):=f_{-1,\sigma}(x)=\dfrac{\beta}{\beta^2+x^2}.
	\end{align*} 
	Let $\Delta\geq 0.02$ be a real number. Then there is a pair of real entire functions $m^{\pm}_{\beta,\Delta}:\mathbb{C}\to\mathbb{C}$ satisfying the following properties:
	\begin{itemize}
		\item[(i)] For $x\in\mathbb{R}$ we have
		\begin{align}\label{Ext_func_f_1_2}
		0< m^{-}_{\beta,\Delta}(x) \leq h_{\beta}(x) \leq  m^{+}_{\beta,\Delta}(x) \ll \dfrac{1}{\beta(1+x^2)}.
		\end{align}
		Moreover, for any complex number $z=x+iy$ we have
		\begin{align}\label{Ext_func_f_21_2}  
		\big|m^{-}_{\beta,\Delta}(z)\big|\ll \dfrac{\beta\Delta^{2}e^{2\pi\Delta|y|}}{(1+\Delta|z|)},  
		\end{align}
		and
		\begin{align}\label{Ext_func_f_22_2}
		\big|m^{+}_{\beta,\Delta}(z)\big|\ll \dfrac{\Delta^{2}e^{2\pi\Delta|y|}}{\beta(1+\Delta|z|)}.  
		\end{align}
		
		\smallskip
		
		\item[(ii)] The Fourier transforms of $m^{\pm}_{\beta,\Delta}$ are even continuous functions supported on the interval $[-\Delta,\Delta]$. For $0\leq \xi<\Delta$ these are given by
		\begin{align} \label{27_9_1:55am_2}
		\widehat{m}_{\beta,\Delta}^{\pm}(\xi) = \pi \left(\dfrac{e^{2\pi\beta(\Delta - \xi)}-e^{-2\pi\beta(\Delta-\xi)}}{\left(e^{\pi\beta\Delta}\mp e^{-\pi\beta\Delta}\right)^2}\right).
		\end{align}
	\end{itemize}
\end{lemma}

\bigskip

 \section{Asymptotic analysis}
In order to prove Theorem \ref{teo_prin}, we shall first apply the Guinand-Weil explicit formula to the extremal functions and then perform a careful asymptotic analysis of the terms appearing in the process. We use this in the representation lemma and finally optimize the support of some Fourier transforms resulting from the previous analysis to get the desired result. We highlight that one of the main technical difficulties of our proof, when compared with results in \cite{CCM2, CChi, CF, CS}, is in the analysis of the sums over prime powers. To obtain the exact asymptotic behavior of such tough terms we shall need explicit formulas for the Fourier transforms of these extremal functions. In Appendix A (the last section) we collect some technical results that will be needed.

\smallskip 

Let $m\geq 1$ be an integer, and $c>0$, $\Delta\geq 0.02$ and $\hh<\sigma<1$ be real numbers such that $(1-\sigma)^2\pi\Delta\geq c$. Let $h^{\pm}_{\Delta}(s)$ be any of the six extremal functions referred to in Lemmas \ref{Ext_func_f}, \ref{Rep_lem_f_{-1}_2} and \ref{Ext_func_f_{-1}}, and let $t\in\R$. As explained in the previous section, we replace each one of the functions $f_{\sigma}, f_{2m+1,\sigma}$ and $f_{-1,\sigma}$ by its extremal functions in Lemma \ref{lem_rep}. This means that we must bound the sum $h^{\pm}_{\Delta}(t-\gamma)$. If we consider the function $h_t(s):=h^{\pm}_{\Delta}(t-s)$, then $\widehat{h_t}(\xi)=\widehat{h}^{\pm}_{\Delta}(-\xi)e^{-2\pi i\xi t}$. It follows from \eqref{Ext_func_f_1}, \eqref{Ext_func_f_21}, \eqref{Ext_func_f_22}, \eqref{Ext_func_f_1_2_2}, \eqref{bound_g_2m+1_complex}, \eqref{Ext_func_f_1_2}, \eqref{Ext_func_f_21_2}, \eqref{Ext_func_f_22_2} and an application of the Phragm\'{e}n-Lindel\"{o}f principle that $|h_t(s)|\ll(1+|s|)^{-2}$ when $|\re{s}|\to\infty$ in the strip $|\im{s}|\leq 1$. Therefore, the function $h_t(s)$ satisfies the hypotheses of Lemma \ref{Exp_for_L}. By the generalized Riemann hypothesis and the fact that $\widehat{h}^{\pm}_{\Delta}$ are even functions we obtain that
\begin{align} 
\begin{split} \label{Sumall_1}
\displaystyle\sum_{\gamma}h^{\pm}_{\Delta}(t-\gamma) & = \frac{\log N}{2\pi}\widehat{h}^{\pm}_{\Delta}(0) + \frac{1}{\pi}\sum_{j=1}^d\int_{-\infty}^\infty h^{\pm}_{\Delta}(t-u)\,{\rm Re}\,\frac{\Gamma_\mathbb R'}{\Gamma_\mathbb R}\left(\hh+\mu_j+iu\right)\d u \\
& \ \ \  -\frac{1}{2\pi}\sum_{n=2}^\infty\frac{1}{\sqrt{n}}\widehat{h}^{\pm}_{\Delta}\bigg(\frac{\log n}{2\pi}\bigg)\Big(\Lambda_{\pi}(n)\, e^{-it\log n}+\overline{\Lambda_{\pi}(n)}\, e^{it\log n}\Big),
\end{split}
\end{align}
where the sum runs over all values of $\gamma$ such that $L\big(\hh+i\gamma,\pi\big)=0$, counted with multiplicity. We now proceed to analyze asymptotically each term on the right-hand side of \eqref{Sumall_1}. 

\subsection{First term} The first is given by \eqref{Ext_func_G_41}, \eqref{27_9_1:55am_2_2} and \eqref{27_9_1:55am_2}.

\subsection{Second term} 
We first examine the functions $g^{\pm}_{\sigma,\Delta}$. It follows from \eqref{Ext_func_f_1}, for any $x\neq 0$, that
	\[
	-\dfrac{1}{x^2}\ll g^{-}_{\sigma,\Delta}(x)\leq f_{\sigma}(x)\ll \dfrac{1}{x^2}.
	\]
Hence, from \eqref{Ext_func_f_21}, we deduce 
\[
|g^{-}_{\sigma,\Delta}(x)|\ll\min\Big\{\dfrac{1}{x^2},\Delta^2\Big\}.
\]
Then, using \eqref{New_Stirling} and the fact that $\Delta\geq 0.02$, we see that
	\begin{align} \label{27_9_4:07am}
	\begin{split} \dfrac{1}{\pi}\int_{-\infty}^\infty g^{-}_{\sigma,\Delta}(t-u)\,&{\rm Re}\,\frac{\Gamma_\mathbb R'}{\Gamma_\mathbb R}\left(\hh+\mu_j+iu\right)\d u \\
	& = \dfrac{1}{2\pi}\int_{-\infty}^{\infty}g^{-}_{\sigma,\Delta}(t-u)\log\left|\hh+\mu_j+iu \right|\d u + O(\Delta^2)  \\
	& = \dfrac{1}{2\pi}\int_{-\infty}^{\infty}g^{-}_{\sigma,\Delta}(u)\big\{\log(|\mu_j+it|+3) +O(\log(|u|+2))\big\}\d u  + O(\Delta^2) \\
	& = \dfrac{\log(|\mu_j+it|+3)}{2\pi}\,\widehat{g}^{-}_{\sigma,\Delta}(0) + O(\Delta^2).
	\end{split}
	\end{align}
Similarly, the relation
\[
|g^{+}_{\sigma,\Delta}(x)|\ll\Omega(\sigma)\min\Big\{\dfrac{1}{x^2},\Delta^2\Big\}
\]
implies that
\begin{align} \label{27_9_1:42am_2}
\begin{split} \int_{-\infty}^\infty g^{+}_{\sigma,\Delta}(t-u)\,{\rm Re}\,\frac{\Gamma_\mathbb R'}{\Gamma_\mathbb R}\left(\hh+\mu_j+iu\right)\d u = \dfrac{\log(|\mu_j+it|+3)}{2\pi}\,\widehat{g}^{+}_{\sigma,\Delta}(0) + O(\Omega(\sigma)\Delta^2).
\end{split}
\end{align}
We next examine the functions $g^{\pm}_{2m+1,\sigma,\Delta}$. Using \eqref{New_Stirling} and \eqref{Ext_func_f_1_2_2} we obtain 
\begin{align} \label{28_9_8:55am}
\int_{-\infty}^\infty g^{\pm}_{2m+1,\sigma,\Delta}(t-u)\,{\rm Re}\,\frac{\Gamma_\mathbb R'}{\Gamma_\mathbb R}\left(\hh+\mu_j+iu\right)\d u & = \dfrac{\log(|\mu_j+it|+3)}{2\pi}\,\widehat{g}^{\pm}_{2m+1,\sigma,\Delta}(0) + O_m(1).
\end{align}
Finally, we examine the functions $m_{\beta,\Delta}^{\pm}$. If $0 < \beta < \frac12$ and $|x| \geq 1$ then
$$h_{\beta}(x) = \frac{\beta}{\beta^2 + x^2} \leq \frac{1}{1 + x^2}.$$
Hence we get from \eqref{Ext_func_f_1_2} that
\begin{align*}
\begin{split}
0 & \leq \int_{-\infty}^{\infty}m^{-}_{\beta,\Delta}(x)\,\log(2+|x|) \,\dx \\
& \leq \int_{-\infty}^{\infty} h_{\beta}(x) \log (2 + |x|)\,\dx  = \int_{-1}^{1}h_{\beta}(x) \log (2 + |x|) \,\dx + \int_{|x| \ge 1}h_{\beta}(x) \log (2 + |x|) \,\dx = O(1),
\end{split} 
\end{align*}
and using \eqref{New_Stirling} we get
\begin{align} \label{28_9_3:32am} \dfrac{1}{\pi}\int_{-\infty}^\infty m^{-}_{\beta,\Delta}(t-u)\,{\rm Re}\,\frac{\Gamma_\mathbb R'}{\Gamma_\mathbb R}\left(\hh+\mu_j+iu\right)\d u & = \dfrac{\log(|\mu_j+it|+3)}{2\pi}\,\widehat{m}^{-}_{\beta,\Delta}(0) + O(1).
\end{align}
Similarly, \eqref{New_Stirling} and \eqref{Ext_func_f_1_2} imply
\begin{align} \label{28_9_3:33am} \dfrac{1}{\pi}\int_{-\infty}^\infty m^{+}_{\beta,\Delta}(t-u)\,{\rm Re}\,\frac{\Gamma_\mathbb R'}{\Gamma_\mathbb R}\left(\hh+\mu_j+iu\right)\d u & = \dfrac{\log(|\mu_j+it|+3)}{2\pi}\,\widehat{m}^{+}_{\beta,\Delta}(0) + O\bigg(\dfrac{1}{\beta}\bigg).
\end{align}

\subsection{Third term} We will make use of the explicit formula for the Fourier transforms of the extremal functions. If we write $x=e^{2\pi\Delta}$, since these Fourier transforms are supported on the interval $[-\Delta,\Delta]$, the third term is a sum that only runs for $2\leq n \leq x$. We start by examining the functions $g^{\pm}_{\sigma,\Delta}$. Observe first that 
\begin{align}\label{Pf_Lemma8_sum_primes_eq1}
\sum_{k\neq 0} \frac{|k+1|}{|\xi+k\Delta|} e^{-2\pi|\xi+k\Delta|}\ll e^{-2\pi \Delta},
\end{align}
when $0 < \xi < \Delta$. Using \eqref{1L_fun_eq3}, \eqref{27_9_1:55am}, \eqref{Pf_Lemma8_sum_primes_eq1} and the prime number theorem we find that 
\begin{align*}
\Bigg|\frac{1}{2\pi}\sum_{n=2}^\infty\frac{1}{\sqrt{n}}\widehat{g}^{\pm}_{\sigma,\Delta}&\bigg(\frac{\log n}{2\pi}\bigg)\Big(\Lambda_{\pi}(n)\, e^{-it\log n}+\overline{\Lambda_{\pi}(n)}\, e^{it\log n}\Big)\Bigg| \\
& \leq 2\,d\, \sum_{n\leq x}\frac{\Lambda(n)}{\sqrt{n}}\Bigg|\sum_{k=-\infty}^{\infty}(\pm 1)^k\frac{(k+1)}{|\log nx^k|}\Big(e^{-|\log nx^k|(\sigma-\frac12)}- e^{-|\log nx^k|}\Big)\Bigg|\\
& \leq 2\,d\, \sum_{n\leq x}\frac{\Lambda(n)}{\sqrt{n}}\Bigg|\sum_{k=-\infty}^{\infty}(\pm 1)^k\frac{(k+1)e^{-|\log nx^k|(\sigma-\frac12)}}{|\log nx^k|}\Bigg| + O(d).
\end{align*}
It is now convenient to split the inner sum in the ranges $k\geq 0$ and $k \leq -2$, and regroup them as
\begin{align}
\begin{split} \label{3_10_8:11pm}
\Bigg|\frac{1}{2\pi}\sum_{n=2}^\infty\frac{1}{\sqrt{n}}\widehat{g}^{\pm}_{\sigma,\Delta}&\bigg(\frac{\log n}{2\pi}\bigg)\Big(\Lambda_{\pi}(n)\, e^{-it\log n}+\overline{\Lambda_{\pi}(n)}\, e^{it\log n}\Big)\Bigg|\\
& \leq 2\,d\,\sum_{n \leq x}\!\dfrac{\Lambda(n)}{\sqrt{n}}\left|\sum_{k=0}^{\infty} (\pm 1)^{k}\left(\!\frac{k+1}{(\log nx^k) \,(nx^k)^{\sigma - \frac12}} - \frac{k+1}{\big( \log \frac{x^{k+2}}{n}\big) \big(\frac{x^{k+2}}{n}\big)^{\sigma - \frac12}}\!\right)\right| + O(d).
\end{split}
\end{align} 
For the function $\widehat{g}^-_{\sigma,\Delta}$, using Appendices \textbf{A.1} and \textbf{A.2} in \eqref{3_10_8:11pm} we obtain that
\begin{align} 
\begin{split} \label{27_9_4:01am} 
\Bigg|\frac{1}{2\pi}\sum_{n=2}^\infty\frac{1}{\sqrt{n}}\widehat{g}^{-}_{\sigma,\Delta}&\bigg(\frac{\log n}{2\pi}\bigg)\Big(\Lambda_{\pi}(n)\, e^{-it\log n}+\overline{\Lambda_{\pi}(n)}\, e^{it\log n}\Big)\Bigg| \\
& \leq 2\,d\,\sum_{n \leq x}\!\dfrac{\Lambda(n)}{\sqrt{n}}\Bigg(\dfrac{1}{n^{\sigma-\frac{1}{2}}\log n} - \dfrac{n^{\sigma-\frac{1}{2}}}{(2\log x - \log n)x^{2\sigma-1}}\Bigg) + O(d) \\
& = \dfrac{d\,(2\sigma-1)}{\sigma(1-\sigma)}\dfrac{e^{(2-2\sigma)\pi\Delta}}{\pi\Delta} +  O_{c}\left(\dfrac{d\,e^{(2-2\sigma)\pi\Delta}}{(1-\sigma)^{2}\Delta^{2}}\right).
\end{split}
\end{align} 
For the function $\widehat{g}^{+}_{\sigma,\Delta}$, we isolate the term $k=0$ and using Appendices \textbf{A.2} and \textbf{A.3} in \eqref{3_10_8:11pm} we get \begin{align} 
\begin{split} \label{27_9_4:01am_2} 
\Bigg|\frac{1}{2\pi}\sum_{n=2}^\infty\frac{1}{\sqrt{n}}\widehat{g}^{+}_{\sigma,\Delta}&\bigg(\frac{\log n}{2\pi}\bigg)\Big(\Lambda_{\pi}(n)\, e^{-it\log n}+\overline{\Lambda_{\pi}(n)}\, e^{it\log n}\Big)\Bigg| \\
& \leq \dfrac{d\,(2\sigma-1)}{\sigma(1-\sigma)}\dfrac{e^{(2-2\sigma)\pi\Delta}}{\pi\Delta} +  O_{c}\left(\dfrac{d\,e^{(2-2\sigma)\pi\Delta}}{(\sigma-\frac{1}{2})(1-\sigma)^{2}\Delta^{2}}\right).
\end{split}
\end{align} 
We next examine the case $g^{\pm}_{2m+1,\sigma,\Delta}$. As we did in the previous case, using \eqref{1L_fun_eq3}, \eqref{27_9_1:55am_2_2_2}, \eqref{Pf_Lemma8_sum_primes_eq1} and the prime number theorem it follows that 
\begin{align*}
\Bigg|\frac{1}{2\pi}\sum_{n=2}^\infty\frac{1}{\sqrt{n}}&\widehat{g}^{\pm}_{2m+1,\sigma,\Delta}\bigg(\frac{\log n}{2\pi}\bigg)\Big(\Lambda_{\pi}(n)\, e^{-it\log n}+\overline{\Lambda_{\pi}(n)}\, e^{it\log n}\Big)\Bigg|\\
& \leq d\,(2m)!\,\sum_{n \leq x}\!\dfrac{\Lambda(n)}{\sqrt{n}}\left|\sum_{k=0}^{\infty} (\pm 1)^{k}\left(\!\frac{k+1}{(\log nx^k)^{2m+2} \,(nx^k)^{\sigma - \frac12}} - \frac{k+1}{\big( \log \frac{x^{k+2}}{n}\big)^{2m+2} \big(\frac{x^{k+2}}{n}\big)^{\sigma - \frac12}}\!\right)\right| + O_m(d).
\end{align*} 
We isolate the term $k=0$ and using Appendices \textbf{A.2} and \textbf{A.3} we get
\begin{align} 
\begin{split} \label{27_9_4:01am_2_2_2} 
\Bigg|\frac{1}{2\pi}\sum_{n=2}^\infty\frac{1}{\sqrt{n}}\widehat{g}^{\pm}_{2m+1,\sigma,\Delta}\bigg(\frac{\log n}{2\pi}\bigg)&\Big(\Lambda_{\pi}(n)\, e^{-it\log n}+\overline{\Lambda_{\pi}(n)}\, e^{it\log n}\Big)\Bigg| \\
& \leq \dfrac{d\,(2m)!\,(2\sigma-1)}{\sigma(1-\sigma)}\dfrac{e^{(2-2\sigma)\pi\Delta}}{(2\pi\Delta)^{2m+2}} +  O_{m,c}\left(\dfrac{d\,e^{(2-2\sigma)\pi\Delta}}{(1-\sigma)^{2}\Delta^{2m+3}}\right)+ O_m(d).
\end{split}
\end{align} 
We finally examine the case $m^{\pm}_{\beta,\Delta}$. Note that in this case we have $(\hh-\beta)^{2} \pi\Delta \geq c$. Using the fact that  $\widehat{m}^{\pm}_{\beta,\Delta}$ are nonnegative (see \eqref{27_9_1:55am_2}), by \eqref{1L_fun_eq3} and Appendix \textbf{A.4} we have that
\begin{align} \begin{split} \label{28_09_4:51am}
\Bigg|\frac{1}{2\pi}\sum_{n=2}^\infty\frac{1}{\sqrt{n}}\widehat{m}^{\pm}_{\beta,\Delta}\bigg(\frac{\log n}{2\pi}\bigg)&\Big(\Lambda_{\pi}(n)\, e^{-it\log n}+\overline{\Lambda_{\pi}(n)}\, e^{it\log n}\Big)\Bigg| \\
& \leq \frac{d}{(e^{\pi\beta\Delta}\mp e^{-\pi\beta\Delta})^2}\displaystyle\sum_{n\leq x}\dfrac{\Lambda(n)}{\sqrt{n}}\left(\dfrac{x^{\beta}}{n^{\beta}}-\dfrac{n^\beta}{x^{\beta}}\right) \\
& \leq \frac{2\,d\,\beta\,e^{(1-2\beta)\pi\Delta}}{(\frac{1}{4}-\beta^2)(1\mp e^{-2\pi\beta\Delta})^2} + O_c\left(\dfrac{d\,\beta\,e^{(1-2\beta)\pi\Delta}}{(\frac{1}{2}-\beta)^2\,\Delta\,(1\mp e^{-2\pi\beta\Delta})^2}\right).
\end{split}
\end{align}
Therefore, for the function $\widehat{m}^{-}_{\beta,\Delta}$ we obtain in \eqref{28_09_4:51am} that 
\begin{align}
\begin{split}  \label{28_9_4:56am}
\Bigg|\frac{1}{2\pi}\sum_{n=2}^\infty\frac{1}{\sqrt{n}}\widehat{m}^{-}_{\beta,\Delta}\bigg(\frac{\log n}{2\pi}\bigg)&\Big(\Lambda_{\pi}(n)\, e^{-it\log n}+\overline{\Lambda_{\pi}(n)}\, e^{it\log n}\Big)\Bigg| \\
& \leq \frac{2d\,\beta\,e^{(1-2\beta)\pi\Delta}}{(\frac{1}{4}-\beta^2)(1+ e^{-2\pi\beta\Delta})^2} + O_c\left(\dfrac{d\,\beta\,e^{(1-2\beta)\pi\Delta}}{(\frac{1}{2}-\beta)^2\,\Delta}\right).
\end{split}
\end{align}
As for the function $\widehat{m}^{+}_{\beta,\Delta}$, considering that 
\begin{equation*}
\frac{1}{\big(1 -e^{-2\pi \beta \Delta}\big)^2} \ll \frac{1}{\big(1 -e^{-\beta}\big)^2} \ll \frac{1}{\beta^2}.
\end{equation*}
we have
\begin{align} \label{28_9_4:57am}
\begin{split} 
\Bigg|\frac{1}{2\pi}\sum_{n=2}^\infty\frac{1}{\sqrt{n}}\widehat{m}^{+}_{\beta,\Delta}\bigg(\frac{\log n}{2\pi}\bigg)&\Big(\Lambda_{\pi}(n)\, e^{-it\log n}+\overline{\Lambda_{\pi}(n)}\, e^{it\log n}\Big)\Bigg| \\
& \leq \frac{2\,d\,\beta\,e^{(1-2\beta)\pi\Delta}}{(\frac{1}{4}-\beta^2)(1-e^{-2\pi\beta\Delta})^2} + O_c\left(\dfrac{d\,e^{(1-2\beta)\pi\Delta}}{\beta\, (\frac{1}{2}-\beta)^2\,\Delta}\right).
\end{split}
\end{align}

\smallskip

\subsection{Final analysis}
\subsubsection{Estimates for $\log|L(s,\pi)|$} We first will prove the upper bound. From Lemma \ref{lem_rep} and \eqref{Ext_func_f_1} we get
\begin{align} 
\begin{split} \label{Principal_relation}
\log|L(\sigma+it,\pi)| \leq \bigg(\dfrac{3}{4}-\dfrac{\sigma}{2}\bigg)\log C(t,\pi) - \dfrac{1}{2}\displaystyle\sum_{\gamma}g^{-}_{\sigma,\Delta}(t-\gamma) %-\displaystyle\sum_{-1<\re\mu_j\leq 0}\log\bigg|\dfrac{\hn+it+\mu}{s+\mu}\bigg| 
+ O(d).	
\end{split}
\end{align}
In other hand, using \eqref{27_9_4:07am} and \eqref{27_9_4:01am} in \eqref{Sumall_1} we obtain
\begin{align} \label{27_9_4:33am} 
\displaystyle\sum_{\gamma}g^{-}_{\sigma,\Delta}(t-\gamma) & \geq \frac{\log C(t,\pi)}{2\pi}\widehat{g}^{-}_{\sigma,\Delta}(0) - \dfrac{d\,(2\sigma-1)}{\sigma(1-\sigma)}\dfrac{e^{(2-2\sigma)\pi\Delta}}{\pi\Delta} + O(d\Delta^2)+ O_c\bigg(\dfrac{d\, e^{(2-2\sigma)\pi\Delta}}{(1-\sigma)^{2}\Delta^2}\bigg).
\end{align}
Then, combining \eqref{Ext_func_G_41}, \eqref{27_9_4:33am} and \eqref{Feb02_4:52pm} in \eqref{Principal_relation}  we get 
\begin{align*} 
\log|L(\sigma+it,\pi)|\leq  
\dfrac{1}{2\pi\Delta}\log\bigg(\dfrac{1+e^{-(2\sigma-1)\pi\Delta}}{1+e^{-2\pi\Delta}}\bigg)\log C(t,\pi) + \dfrac{d\,(2\sigma-1)}{\sigma(1-\sigma)}\dfrac{e^{(2-2\sigma)\pi\Delta}}{2\pi\Delta} +  O_c\bigg(\dfrac{d\,e^{(2-2\sigma)\pi\Delta}}{(1-\sigma)^{2}\Delta^2}\bigg).
\end{align*}
Since $\log\log C(t,\pi)\geq \log\log 3>0.09$, we can choose $\pi\Delta=\log\log C(t,\pi)$. Then
\[
\dfrac{1}{2\pi\Delta}\log\big(1+e^{-2\pi\Delta}\big)\log C(t,\pi) \ll \dfrac{d\,e^{(2-2\sigma)\pi\Delta}}{(1-\sigma)^{2}\Delta^2},
\]
and the desired result follows from \eqref{25_10_1:52am}. The proof of the lower bound is similar, combining \eqref{Ext_func_f_1}, \eqref{Ext_func_G_41}, \eqref{Sumall_1}, \eqref{27_9_1:42am_2}, \eqref{27_9_4:01am_2}, \eqref{Feb02_4:52pm} with Lemma \ref{Rep_lem_f}.

\smallskip

\subsubsection{Estimates for $S_{2m+1,\sigma}(t,\pi)$} Let us first consider the case where $m$ is even. We will prove the upper bound. From Lemma \ref{lem_rep} and \eqref{Ext_func_f_1_2_2} we have that
\begin{align} 
\begin{split} \label{Principal_relation_2_2}
S_{2m+1,\sigma}(t,\pi)\leq\dfrac{1}{2\pi(2m+2)!}\,\bigg(\dfrac{3}{2}-\sigma\bigg)^{2m+2}\log C(t,\pi) -\dfrac{1}{\pi(2m)!}\,\displaystyle\sum_{\gamma}g^{-}_{2m+1,\sigma,\Delta}(t-\gamma)+ O_m(d).
\end{split}
\end{align}
Combining \eqref{27_9_1:55am_2_2}, \eqref{Sumall_1}, \eqref{28_9_8:55am}, \eqref{27_9_4:01am_2_2_2} and \eqref{Feb02_4:52pm} in \eqref{Principal_relation_2_2} we get 
\begin{align}\label{Jan_26_7:03}
S_{2m+1,\sigma}(t,\pi) & \leq \dfrac{\log C(t,\pi)}{(2m)!\,2\pi^2 \Delta} \int_{\sigma}^{\hn}(u-\sigma)^{2m}\log\left(\dfrac{1+ e^{-2\pi(u-\frac12)\Delta}}{1+ e^{-2\pi\Delta}}\right)\du + \dfrac{d\,(2\sigma-1)}{\pi\sigma(1-\sigma)}\dfrac{e^{(2-2\sigma)\pi\Delta}}{(2\pi\Delta)^{2m+2}} + O_m(d) \nonumber \\
& \ \ \ \ + O_{m,c}\left(\dfrac{d\,e^{(2-2\sigma)\pi\Delta}}{(1-\sigma)^{2}\Delta^{2m+3}}\right).
\end{align}
We now choose $\pi \Delta = \log \log C(t,\pi)$. Using \eqref{Feb02_4:52pm} in \eqref{Jan_26_7:03} leads us to
\begin{align*}
S_{2m+1,\sigma}(t,\pi) &  \leq \dfrac{\log C(t,\pi)}{(2m)!\,2\pi^2 \Delta} \int_{\sigma}^{\infty}(u-\sigma)^{2m}\,\log\Big(1+ e^{-2\pi(u-\frac12)\Delta}\Big)\du + \dfrac{d\,(2\sigma-1)}{\pi\sigma(1-\sigma)}\dfrac{e^{(2-2\sigma)\pi\Delta}}{(2\pi\Delta)^{2m+2}} \\
& \,\, \, \, \, \, \, + O_{m,c}\left(\dfrac{d\,e^{(2-2\sigma)\pi\Delta}}{(1-\sigma)^{2}\Delta^{2m+3}}\right).
\end{align*}
Finally, taking into account that
\begin{align*}
\int_{\sigma}^{\infty}(u-\sigma)^{2m}\log\Big(1+ e^{-2\pi(u-\frac12)\Delta}\Big)\,\du  =   \frac{(2m)!}{(2\pi\Delta)^{2m+1}}\sum_{k=1}^{\infty} \frac{(-1)^{k+1}\,e^{-2k\pi (\sigma - \frac12)\Delta}}{k^{2m+2}},
\end{align*}
we obtain the desired result. The proof of the lower bound is obtained similarly, combining \eqref{Ext_func_f_1_2_2}, \eqref{27_9_1:55am_2_2}, \eqref{Sumall_1}, \eqref{28_9_8:55am}, \eqref{27_9_4:01am_2_2_2}, \eqref{Feb02_4:52pm} with Lemma \ref{lem_rep}. When $m$ is odd, the proof is similar, since only the roles of the majorant $g^{+}_{2m+1,\sigma\Delta}$ and minorant $g^{-}_{2m+1,\sigma\Delta}$ are interchanged due to the presence of the factor $(-1)^m$ in Lemma \ref{lem_rep}.

\subsubsection{Estimates for $S_{-1,\sigma}(t,\pi)$} 
Let us first prove the lower bound. From Lemma \ref{lem_rep} and \eqref{Ext_func_f_1_2} we have
\begin{align} 
\begin{split} \label{Principal_relation_2}
-\dfrac{\log C(t,\pi)}{2\pi} + \dfrac{1}{\pi}\displaystyle\sum_{\gamma}m^{-}_{\beta,\Delta}(t-\gamma)+ O(d) \leq S_{-1,\sigma}(t,\pi).
\end{split}
\end{align}
Combining \eqref{27_9_1:55am_2}, \eqref{Sumall_1}, \eqref{28_9_3:32am}, \eqref{28_9_4:56am} in \eqref{Principal_relation_2} we deduce that 
\begin{align*}
\begin{split}
S_{-1,\sigma}(t,\pi) & \geq -\frac{\log C(t,\pi)}{\pi} \left(\frac{e^{-2 \pi \beta \Delta}}{1 + e^{-2\pi \beta \Delta}}\right) - \dfrac{2\,d\,\beta\,e^{(1-2\beta)\pi\Delta}}{\pi(\frac{1}{4}-\beta^2)\big(1+ e^{-2\pi\beta\Delta}\big)^2} + O_c\left(\dfrac{d\,\beta\,e^{(1-2\beta)\pi\Delta}}{(\frac{1}{2}-\beta)^2\Delta}\right)+O(d).
\end{split}
\end{align*}
We now choose $\pi \Delta = \log \log C(t,\pi)$. Recalling that $\beta = \sigma - \frac12$, by \eqref{Feb02_4:52pm} this choice yields
\begin{align*}
\!\!S_{-1,\sigma}(t,\pi) & \geq - \frac{(\log C(t,\pi))^{2 - 2 \sigma}}{\pi} \!\left(\! \frac{1}{\big( 1 \!+\! (\log C(t,\pi))^{1 - 2 \sigma}\big)} \!+\!  \frac{d\,(2\sigma -1)}{\sigma (1- \sigma)\big( 1 \!+\! (\log C(t,\pi))^{1 - 2 \sigma}\big)^2}\!\right)  \\
& \ \ \  + O_c\left(\dfrac{d\,(\sigma - \frac12) (\log C(t,\pi))^{2 - 2 \sigma}}{(1 - \sigma)^2 \log \log C(t,\pi)}\right).
\end{align*}
Observe that this estimate is actually slightly stronger than the one we proposed in Theorem \ref{teo_prin}. For the proof of the upper bound, as before, combining \eqref{Ext_func_f_1_2}, \eqref{27_9_1:55am_2}, \eqref{Sumall_1}, \eqref{28_9_3:33am}, \eqref{28_9_4:57am}, \eqref{Feb02_4:52pm} with Lemma \ref{lem_rep}, and choosing $\pi\Delta=\log\log C(t,\pi)$ we obtain that
\begin{align}\label{Fev02_4:00pm}
\begin{split}
S_{-1,\sigma}(t,\pi) & \leq \frac{(\log C(t,\pi))^{2 - 2 \sigma}}{\pi} \left( \frac{1}{\big( 1 - (\log C(t,\pi))^{1 - 2 \sigma}\big)} +  \frac{d\,(2\sigma -1)}{\sigma (1- \sigma)\big( 1 - (\log C(t,\pi))^{1 - 2 \sigma}\big)^2}\right) \\
& \ \ \ \  + O_c\left(\dfrac{d\,(\log C(t,\pi))^{2 - 2 \sigma}}{(\sigma - \frac12)(1 - \sigma)^2 \log \log C(t,\pi)}\right).
\end{split}
\end{align}
Finally, note that if we write $\theta = \log C(t,\pi)$, then $\theta \geq \log 3 >1$, and therefore 
\begin{equation*}
\left(1 -  \frac{1}{\big( 1 - \theta^{1 - 2 \sigma}\big)^2}\right) \ll \frac{\theta^{1 - 2 \sigma}}{\big( 1 - \theta^{1 - 2 \sigma}\big)^2} \ll \frac{1}{(\sigma - \frac12)^2 (\log \theta)^2} \ll \frac{1}{(\sigma - \frac12)^2 (\log \theta)}.
\end{equation*}
By applying this bound in \eqref{Fev02_4:00pm}, we  obtain the desired result.

\medskip

\section{Interpolation tools}
In order to bound the functions $S_{2m,\sigma}(t,\pi)$ when $m\geq0$ is an integer, we follow a different argument to the case of $S_{2m+1,\sigma}(t,\pi)$. Although we can obtain a representation as in Lemma \ref{lem_rep} (see \cite[Lemma 7]{CChiM}), it is unknown to find extremal majorants and minorants of exponential type for the associated functions in the representation. Therefore, we adopt a different approach based on an interpolation argument. We follow the same outline as in \cite[Section 6]{CChiM}, where similar functions associated with the Riemann zeta-function were studied. Here we present the necessary changes to adapt the proof in \cite{CChiM} for our family of entire $L$-functions. The main change consists in the suitable use of the mean value theorem, since the analytic conductor is not sufficiently smooth. 
\smallskip

 Since we assume the generalized Riemann hypothesis and $\hh<\sigma<1$, we have that $S'_{2m+1,\sigma}(t,\pi)=S_{2m,\sigma}(t,\pi)$ and $S'_{2m,\sigma}(t,\pi)=S_{2m-1,\sigma}(t,\pi)$ for all $t\in\R$. For $n\geq 0$ we consider the following functions
\begin{equation*}
l_{n,\sigma}(t):=\dfrac{(\log C(t,\pi))^{2-2\sigma}}{(\log\log C(t,\pi))^{n}} \hspace{0.5cm} r_{n,\sigma}(t):=\dfrac{d\,(\log C(t,\pi))^{2-2\sigma}}{(1-\sigma)^2(\log\log C(t,\pi))^{n}}.
\end{equation*}

\smallskip

\subsubsection{Estimates for $S_{0,\sigma}(t,\pi)$}

Let $c>0$ be a given real number. In the range  $(1-\sigma)^{2} \geq \frac{c/16}{\log\log C(t,\pi)}$ we have already shown that
\begin{align}\label{interpol_2_Feb07}
-M_{1,\sigma}^-(t)\,\ell_{2,\sigma}(t)+O_{c}(r_{3,\sigma}(t))\leq S_{1,\sigma}(t,\pi)\leq M_{1,\sigma}^+(t)\,\ell_{2,\sigma}(t)+O_{c}(r_{3,\sigma}(t)),
\end{align} 
and that
\begin{align} \label{interpol_1_Feb07}
-M_{-1,\sigma}^-(t)\,\ell_{0,\sigma}(t)+O_{c}(r_{1,\sigma}(t))\leq S_{-1,\sigma}(t,\pi).
\end{align}
Let $(\sigma,t)$ be such that $(1-\sigma)^{2}\geq\frac{c}{\log\log C(t,\pi)}$. By Appendix \textbf{A.5} we have that in the set $\{(\sigma,\mu); \,t-25\leq\mu\leq t+25\}$, estimates \eqref{interpol_2_Feb07} and \eqref{interpol_1_Feb07} hold. Then, by the mean value theorem and \eqref{interpol_1_Feb07}, we obtain for $0 \leq h \leq 25$, 
\begin{align}\label{Feb07_4:02pm}
\begin{split} 
S_{0,\sigma}(t,\pi) - S_{0,\sigma}(t-h,\pi)  = h\, S_{-1,\sigma}(t_h^*,\pi)  & \geq -h\,M_{-1,\sigma}^-(t_h^*)\,\ell_{0,\sigma}(t_h^*)+h\,O_c(r_{1,\sigma}(t_h^*)) \\
& = -h\,M_{-1,\sigma}^-(t_h^*)\,\ell_{0,\sigma}(t_h^*)+h\,O_c(r_{1,\sigma}(t)) ,
\end{split}
\end{align}
where $t^*_h$ is a suitable point in the segment connecting $t-h$ and $t$. We claim that
\begin{align} \label{29_9_5:37pm}
|M_{-1,\sigma}^-(t)\,\ell_{0,\sigma}(t)-M_{-1,\sigma}^-(t^{*}_h)\,\ell_{0,\sigma}(t^{*}_h)|\ll d\,\mu_{d,\sigma},
\end{align}
where $\mu_{d,\sigma}=(2\sigma-1)d + 1$. In order to prove this, we define the function
$$
g_1(x)=\frac{1}{\pi}\left( \frac{1}{1 + x^{1 - 2\sigma}} + \frac{d\,(2\sigma -1)}{\sigma(1 - \sigma)}\right) x^{2 - 2\sigma}.
$$
Note that $|g'_1(x)|\ll \mu_{\dd}$ for $x>1$, and $g_1(\log C(t,\pi))=M_{-1,\sigma}^-(t)\,\ell_{0,\sigma}(t)$. The mean value theorem applied to the functions $g_1$ and the logarithm imply that
\begin{align}
\begin{split}  \label{2_10_2:07am}
\big|g_1(\log C(t,\pi))-g_1(\log C(t^{*}_h,\pi))\big| & \ll \mu_{d,\sigma}|\log C(t,\pi)-\log C(t^{*}_h,\pi)| \\
& \leq \mu_{d,\sigma}\displaystyle\sum_{j=1}^{d}\big|\log(|\mu_j+it|+3)-\log(|\mu_j+it^{*}_h|+3)\big| \\
& \ll \mu_{d,\sigma}\displaystyle\sum_{j=1}^{d}\big||\mu_j+it|-|\mu_j+it^{*}_h|\big| \\
& \leq \mu_{d,\sigma}\displaystyle\sum_{j=1}^{d}|t-t^{*}_h|\ll d\,\mu_{d,\sigma}.
\end{split}
\end{align}
We thus obtain \eqref{29_9_5:37pm}, and using \eqref{Feb02_4:52pm} we have that
\begin{equation}\label{Feb08_4:05pm}
\big|M_{-1,\sigma}^-(t)\,\ell_{0,\sigma}(t)-M_{-1,\sigma}^-(t_h^*)\,\ell_{0,\sigma}(t_h^*)\big|\ll \mu_{d,\sigma}\, r_{1,\sigma}(t).
\end{equation}
From \eqref{Feb07_4:02pm} and \eqref{Feb08_4:05pm} it follows that 
\begin{equation}\label{Feb08_4:09pm}
S_{0,\sigma}(t,\pi) - S_{0,\sigma}(t-h,\pi) \geq -h\,M_{-1,\sigma}^-(t)\,\ell_{0,\sigma}(t)+h\,O_c(\mu_{d,\sigma}\,r_{1,\sigma}(t)).
\end{equation}
Let $\nu = \nu_{\sigma}(t)$ be a real-valued function such that $0 < \nu \leq 25$. For a fixed $t$, we integrate \eqref{Feb08_4:09pm} with respect to the variable $h$ to obtain
\begin{align*}
S_{0,\sigma}(t,\pi) & \geq \frac{1}{\nu} \int_{0}^{\nu} S_{0,\sigma}(t-h,\pi)\,\d h  - \frac{1}{\nu} \left( \int_{0}^{\nu}h\,\d h\right)\,M_{-1,\sigma}^-(t)\,\ell_{0,\sigma}(t) + \frac{1}{\nu} \left(\int_{0}^{\nu} h\,\d h \right) O_c(\mu_{d,\sigma}\,r_{1,\sigma}(t))\\
& = \frac{1}{\nu} \big( S_{1,\sigma}(t,\pi) - S_{1,\sigma}(t - \nu,\pi)\big) - \dfrac{\nu}{2}\,M_{-1,\sigma}^-(t)\,\ell_{0,\sigma}(t) + O_c(\nu\,\mu_{d,\sigma}\,r_{1,\sigma}(t)).
\end{align*}
From \eqref{interpol_2_Feb07} we then get
\begin{align}\label{Feb08_4:36pm}
\begin{split}
S_{0,\sigma}(t,\pi) & \geq \frac{1}{\nu} \Big[\!-M_{1,\sigma}^-(t)\,\ell_{2,\sigma}(t)-M_{1,\sigma}^+(t-\nu)\,\ell_{2,\sigma}(t-\nu)+O_{c}(r_{3,\sigma}(t)) + O_{c}(r_{3,\sigma}(t-\nu))\Big]\\
&  \ \ \ \ \ \ \ \ \ \ - \dfrac{\nu}{2}\,M_{-1,\sigma}^-(t)\,\ell_{0,\sigma}(t) + O_c(\nu\,\mu_{d,\sigma}\,r_{1,\sigma}(t))\\
& = - \Big[M_{1,\sigma}^-(t) + M_{1,\sigma}^+(t)\Big]\,\frac{1}{\nu} \,\ell_{2,\sigma}(t) - \dfrac{\nu}{2}\,M_{-1,\sigma}^-(t)\,\ell_{0,\sigma}(t) + O_{c}\left(\frac{\mu_{d,\sigma}\,r_{3,\sigma}(t)}{\nu}\right) +  O_c(\nu\,\mu_{d,\sigma}\,r_{1,\sigma}(t)),
\end{split}
\end{align}
where the following was used
\begin{align} \label{2_10_2:31am}
\big|M_{1,\sigma}^+(t)\,\ell_{2,\sigma}(t)-M_{1,\sigma}^+(t-\nu)\,\ell_{2,\sigma}(t-\nu)\big|\ll \mu_{d,\sigma}\,r_{3,\sigma}(t).
\end{align}
We now prove \eqref{2_10_2:31am}. For $x>0$ define
$$
g_2(x) = \dfrac{1}{4\pi}\Bigg(\sum_{k=0}^{\infty} \frac{(-1)^k}{(k+1)^2x^{(2\sigma -1)k}} + \frac{d\,(2\sigma -1)}{\sigma(1 - \sigma)}\Bigg)\dfrac{x^{2-2\sigma}}{(\log x)^2}.
$$
Note that $M_{1,\sigma}^+(t)\,\ell_{2,\sigma}(t)=g_2(\log C(t,\pi))$. For each $k\geq 0$ and $x\geq \log 3>1$ put
$$f_{k}(x)=\dfrac{1}{x^{(2\sigma-1)k}}\dfrac{x^{2-2\sigma}}{(\log x)^2}=\dfrac{x^{(k+1)(1-2\sigma)+1}}{(\log x)^2}.$$
Then, for $x>y\geq \log 3$ using the mean value theorem, we have that
\begin{align}\label{Jan_29_12:01}
\begin{split} 
|g_2(x)-g_2(y)|&\ll \displaystyle\sum_{k=0}^{\infty}\dfrac{1}{(k+1)^2}\big|f_{k}(x)-f_{k}(y)\big| + \frac{d\,(2\sigma-1)}{(1 - \sigma)}|f_0(x)-f_0(y)| \\
&= |x-y|\bigg(\displaystyle\sum_{k=0}^{\infty}\dfrac{1}{(k+1)^2}\big|f^{'}_{k}(\xi_k)\big|+\dfrac{d\,(2\sigma-1)}{1-\sigma}\big|f'_{0}(\xi)\big|\bigg) \\
& \ll |x-y|\bigg( \displaystyle\sum_{k=0}^{\infty}\dfrac{((k+1)(2\sigma-1)+1)}{(k+1)^2\,\xi_k^{(k+1)(2\sigma-1)}(\log \xi_k)^2}+ \dfrac{d\,(2\sigma-1)}{1-\sigma}\bigg),
\end{split}
\end{align}
where $\xi_k,\xi\in\,]y,x[$ for each $k \geq 0$. Observe now that by the mean value theorem
\begin{align*}\label{Jan_29_12:02}
\begin{split} 
\displaystyle\sum_{k=0}^{\infty}\dfrac{((k+1)(2\sigma-1)+1)}{(k+1)^2\,\xi_k^{(k+1)(2\sigma-1)}(\log \xi_k)^2}& \leq \displaystyle\sum_{k=0}^{\infty}\dfrac{((k+1)(2\sigma-1)+1)}{(k+1)^2\,y^{(k+1)(2\sigma-1)}(\log y)^2} \\
&  \ll \dfrac{1}{(\log y)^2}\left[\displaystyle\sum_{k=0}^{\infty}\dfrac{2\sigma-1}{(k+1)y^{(k+1)(2\sigma-1)}} + 1 + \dfrac{d\,(2\sigma-1)}{1-\sigma}\right] \\
& \leq \dfrac{1}{(\log y)^2}\left[\displaystyle\sum_{k=0}^{\infty}\dfrac{2\sigma-1}{y^{(k+1)(2\sigma-1)}} + 1 + \dfrac{d\,(2\sigma-1)}{1-\sigma}\right] \ll \dfrac{\mu_{d,\sigma}}{(1-\sigma)(\log y)^2}.
\end{split}
\end{align*}
Then, in \eqref{Jan_29_12:01}, by using a similar idea as in \eqref{2_10_2:07am}, we obtain 
\begin{align*}
\begin{split} 
\Big|g_2(\log C(t,\pi))-g_2(\log C(t-\nu,\pi))\Big|& \ll \frac{\mu_{d,\sigma}}{(1 - \sigma)}\dfrac{|\log C(t,\pi)-\log C(t-\nu,\pi)|}{(\log \log C(t,\pi))^2} \\
& \ll \dfrac{d\,\mu_{d,\sigma}\,(\log C(t,\pi))^{2-2\sigma}}{(1-\sigma)^2(\log\log C(t,\pi))^{3}}.
\end{split}
\end{align*}
This proves \eqref{2_10_2:31am}. We now choose $\nu = \frac{\lambda_{\sigma}(t)}{\log \log C(t,\pi)}$ in \eqref{Feb08_4:36pm}, where $\lambda_{\sigma}(t)>0$ is a function to be determined. This yields
\begin{align*}
\begin{split}
S_{0,\sigma}(t) & \geq - \left[\Big(M_{1,\sigma}^-(t) + M_{1,\sigma}^+(t)\Big)\,\frac{1}{\lambda_{\sigma}(t)} +  \dfrac{M_{-1,\sigma}^-(t)}{2}\,\lambda_{\sigma}(t)\right] \ell_{1,\sigma}(t) + O_{c}\left(\frac{\mu_{d,\sigma}\,r_{2,\sigma}(t)}{\lambda_{\sigma}(t)}\right) +  O_c(\mu_{d,\sigma}\,\lambda_{\sigma}(t)\,r_{2,\sigma}(t)).
\end{split}
\end{align*}
The optimal $\lambda_{\sigma}(t)$ minimizing the expression in brackets is
\begin{equation}\label{Feb08_5:05pm}
\lambda_{\sigma}(t) = \left(\frac{2 \big(M_{1,\sigma}^-(t) + M_{1,\sigma}^+(t)\big)}{M_{-1,\sigma}^-(t)}\right)^{\hh}.
\end{equation}
and this leads to the bound
\begin{equation}\label{Feb09_8:48am}
S_{0,\sigma}(t) \geq - \Big[2 \big(M_{1,\sigma}^-(t) + M_{1,\sigma}^+(t)\big)\,M_{-1,\sigma}^-(t)\Big]^{\hh}\, \ell_{1,\sigma}(t) + O_{c}\left(\frac{\mu_{d,\sigma}\,r_{2,\sigma}(t)}{\lambda_{\sigma}(t)}\right) +  O_c(\mu_{d,\sigma}\,\lambda_{\sigma}(t)\,r_{2,\sigma}(t)).
\end{equation}
Finally, using some estimates for $H_n(x)$, one can show that $\hh\leq \lambda_{\sigma}(t)\leq 2$, which implies that indeed $0< \nu \leq 25$, and allows us to write  \eqref{Feb09_8:48am} in our originally intended form of
\begin{equation*}
S_{\sigma}(t) \geq - \Big[2 \big(M_{1,\sigma}^-(t) + M_{1,\sigma}^+(t)\big)\,M_{-1,\sigma}^-(t)\Big]^{\hh}\, \ell_{1,\sigma}(t) + O_{c}(\mu_{d,\sigma}\,r_{2,\sigma}(t)).
\end{equation*}
The proof of the upper bound for $S_{0,\sigma}(t)$ follows along the same lines. 

\subsubsection{Estimates for $S_{2m,\sigma}(t,\pi)$} The proof of this estimates follows the same outline in \cite[Subsection 6.1]{CChiM}. The substantial changes in the use of the mean value theorem are similar with \eqref{2_10_2:07am} and \eqref{Jan_29_12:01}.

\bigskip

\section*{Appendix A: Calculus facts}

Throughout this paper we shall encounter the following setting in many situations: let $c>0$ be a given real number, $\hh < \sigma < 1$ and $x \geq 2$ be such that 
\begin{equation*}
(1-\sigma)^{2}\log x \geq c.
\end{equation*}
Let us note that, if $0 \leq \theta_1 , \theta_2$ are real numbers, it follows from the above inequality that 
\begin{equation}\label{Feb02_4:52pm}
(1-\sigma)^{\theta_1} \,(\log x)^{\theta_2} \ll_{c,\theta_1, \theta_2} x^{1-\sigma}.
\end{equation}

\smallskip

\subsection*{A.1}
{\it For $\hh < \sigma < 1$ and $2\leq n \leq x$ we have }
\begin{equation*}
0\leq \sum_{k=0}^{\infty} (-1)^{k}\left(\!\frac{k+1}{(\log nx^k) \,(nx^k)^{\sigma - \frac12}} - \frac{k+1}{\big( \log \frac{x^{k+2}}{n}\big) \big(\frac{x^{k+2}}{n}\big)^{\sigma - \frac12}}\!\right) \leq \dfrac{1}{n^{\sigma-\frac{1}{2}}\log n} - \dfrac{n^{\sigma-\frac{1}{2}}}{(2\log x - \log n)x^{2\sigma-1}}.
\end{equation*}
\begin{proof}
See \cite[Eq. (2.14), (2.16) and Lemma 6]{CC}.

\end{proof}

\subsection*{A.2} {\it Let $c>0$ be a given real and $m\geq0$ be an integer or $m=-\hh$. For $\hh < \sigma < 1$ and $x \geq 2$ such that $(1-\sigma)^{2}\log x \geq c$, we have the following asymptotic behaviors}
\begin{equation*}
\displaystyle\sum_{n\leq x}\dfrac{\Lambda(n)}{n^\sigma(\log n)^{2m+2}} = \dfrac{x^{1-\sigma}}{(1-\sigma)(\log x)^{2m+2}}+ O_{l,c}\left(\dfrac{x^{1-\sigma}}{(1-\sigma)^{2}(\log x)^{2m+3}}\right)
\end{equation*}
and
\begin{equation*}
\dfrac{1}{x^{2\sigma-1}}\displaystyle\sum_{n\leq x}\dfrac{\Lambda(n)}{n^{1-\sigma}(2\log x-\log n)^{2m+2}} = \dfrac{x^{1-\sigma}}{\sigma(\log x)^{2m+2}}+ O_{l,c}\left(\dfrac{x^{1-\sigma}}{(1-\sigma)^{2}(\log x)^{2m+3}}\right). 
\end{equation*}
\begin{proof}
See \cite[Appendix B.1, B.2]{CChiM}.

\end{proof}

\subsection*{A.3} {\it Let $c>0$ be a given real number and $m\geq0$ be an integer. For $\hh < \sigma < 1$ and $x \geq 2$ such that $(1-\sigma)^{2}\log x \geq c$, we have the following asymptotic behavior}
\begin{align*}
\displaystyle\sum_{k=1}^{\infty}\dfrac{k+1}{\big(x^{\sigma-\frac{1}{2}}\big)^k}& \Bigg|\displaystyle\sum_{n\leq x}\Lambda(n) \Bigg(\dfrac{1}{n^{\sigma}(k\log x+\log n)^{2m+2}}-\dfrac{1}{x^{2\sigma-1}\,n^{1-\sigma}((k+2)\log x-\log n)^{2m+2}}\Bigg)\Bigg| \\
& \ \ \ \ll_{c} \dfrac{x^{1-\sigma}}{(1-\sigma)^{2}(\log x)^{2m+3}}.
\end{align*}
{\it Besides, we have that}
\begin{align*}
\displaystyle\sum_{k=1}^{\infty}\dfrac{k+1}{\big(x^{\sigma-\frac{1}{2}}\big)^k}& \Bigg|\displaystyle\sum_{n\leq x}\Lambda(n) \Bigg(\dfrac{1}{n^{\sigma}(k\log x+\log n)}-\dfrac{1}{x^{2\sigma-1}\,n^{1-\sigma}((k+2)\log x-\log n)}\Bigg)\Bigg| \\
& \ \ \ \ll_{m,c} \dfrac{x^{1-\sigma}}{(\sigma-\hh)(1-\sigma)^{2}(\log x)^{2}}.
\end{align*}
\begin{proof}
See \cite[Appendix B.3]{CChiM} for the first result. The proof of the second result follows the same outline.

\end{proof}

\subsection*{A.4}   {\it Let $c>0$ be a given real number. For $0 < \beta < \hh$ and $x \geq 2$ such that $(\hh-\beta)^{2}\log x \geq c$, we have }
\begin{equation*}
\displaystyle\sum_{n\leq x}\dfrac{\Lambda(n)}{\sqrt{n}}\left(\dfrac{x^{\beta}}{n^{\beta}}-\dfrac{n^\beta}{x^{\beta}}\right)= \dfrac{2\beta x^{\frac{1}{2}}}{\frac{1}{4}-\beta^2}+ O_c\left(\dfrac{\beta x^{\frac{1}{2}}}{(\frac{1}{2}-\beta)^2\log x}\right).
\end{equation*}
\begin{proof}
It follows by \cite[Appendix B.4]{CChiM} and the mean value theorem.

\end{proof}

\subsection*{A.5}   {Let $z,w$ be complex numbers such that $|w|\leq 25$. Then}
\[
(\log(|z+w|+3))^{16} \geq \log(|z|+3).
\]	
\begin{proof}
	If $|z|> 25$, then 
	\begin{align*}
	(\log(|z+w|+3))^{16} & \geq \log(|z|-|w|+3)(\log 3)^{15} > 4 \log(|z|-22)\geq \log(|z|+3), 
	\end{align*}
	since $(\log 3)^{15}>4$. On the other hand, if $|z|\leq 25$
	\begin{align*}
	(\log(|z+w|+3))^{16} & \geq (\log 3)^{16} > 4 > \log(28) \geq \log (|z|+3).
	\end{align*}
\end{proof}

\section*{Acknowledgements}
The author thanks Emanuel Carneiro and Vorrapan Chandee for inspiring discussions. The author acknowledges support from FAPERJ - Brazil.

\end{document}